\documentclass[10pt,oneside]{amsart}
\usepackage{amssymb, amscd, amsmath, amsthm, latexsym, hyperref}
\usepackage{pb-diagram, fancyhdr, graphicx, psfrag}
\newcommand{\compactlist}{\begin{list}{$\bullet$}{\setlength{\leftmargin}{1em}}}

\def\zz{{\bf Z}}

\def\ff{{\bf F}}

\def\qq{{\bf Q}}

\def\co{\colon\thinspace}
\def\cs{\mathop{\#}}
\def\caln{\mathcal{N}}

\def\calc{\mathcal{C}}
\def\calm{\mathcal{M}}
\def\calg{\mathcal{G}}
\def\calk{\mathcal{K}}
\def\calp{\mathcal{P}}

\newcommand{\spinc}{\ifmmode{{\mathfrak s}}\else{${\mathfrak s}$\ }\fi}
\newcommand{\spinct}{\ifmmode{{\mathfrak t}}\else{${\mathfrak t}$\ }\fi}
\newcommand{\Spc}{Spin$^c$}
\newcommand{\fig}[2] { \includegraphics[scale=#1]{#2} }

\newtheorem{theorem}{Theorem}
\newtheorem{lemma}[theorem]{Lemma}
\newtheorem{corollary}[theorem]{Corollary}
\newtheorem{prop}[theorem]{Proposition}
\theoremstyle{definition}
\newtheorem{definition}[theorem]{Definition}

\numberwithin{equation}{section}

\begin{document}
\title[Non-splittability of $\qq$--homology cobordism]{Non-splittability of the rational homology cobordism group of 3--manifolds}
\author{Se-Goo Kim}\author{Charles Livingston}
\thanks{This work was supported in part by the National Science Foundation under Grant 1007196 and by the National Research Foundation of Korea (NRF) grant funded by the Korea government (MEST) No.~2011--0012893. \\ \today}

\address{Se-Goo Kim: Department of Mathematics and Research Institute for Basic Sciences, Kyung Hee University, Seoul 130--701, Korea }
\email{sgkim@khu.ac.kr}
\address{Charles Livingston: Department of Mathematics, Indiana University, Bloomington, IN 47405 }
\email{livingst@indiana.edu}


\begin{abstract}  Let $\zz[1/p]$ denote the ring of integers with the prime $p$ inverted.  There  is a canonical homomorphism 
$\Psi\co \hskip-.07in \oplus  \Theta^3_{\zz[1/p]}  \to \Theta^3_{\qq}$, where   $\Theta^3_{R}$ denotes the  three-dimensional smooth $R$--homology cobordism group of $R$--homology spheres and the direct sum is over all prime integers.   Gauge theoretic  methods prove the kernel is infinitely generated.  Here we prove that $\Psi$ is not surjective, with cokernel infinitely generated.  
As a basic example we show that for $p$ and $q$ distinct primes,  there is no rational homology cobordism from the lens space  $L(pq,1) $ to any  $M_p \cs M_q$, where $H_1(M_p) = \zz_{p}$ and $H_1(M_q) = \zz_{q}$.  More subtle examples  include cases in which a cobordism to  such a connected sum exists topologically but not smoothly.  (Conjecturally, such a splitting always exists topologically.)  Further examples can be chosen to represent 2--torsion in $\Theta^3_{\qq}$.   

Let $\calk$ denote the kernel  of $\Theta^3_\qq \to \widehat{\Theta}^{3}_\qq$, where  $\widehat{\Theta}^{3}_\qq$ denotes the topological homology cobordism group.  Freedman proved that $\Theta^3_\zz \subset \calk$.  A corollary of results here is that $\calk / \Theta^3_\zz $ is infinitely generated.  We also  demonstrate the failure in dimension three of splitting theorems that apply to higher dimensional knot concordance groups. \end{abstract}

\maketitle

\section{Introduction.}\label{sectionintroduction}
In~\cite{furuta}, Furata applied instanton theory to  reveal unexpectedly deep structure in the homology cobordism group of smooth homology 3--spheres, $\Theta^3_\zz$.  Here we will use the added algebraic structures associated to Heegaard--Floer theory to identify further complications in the rational cobordism group, $\Theta^3_\qq$.

As a simple example, an application of Lisca's rational homology cobordism classification of lens spaces~\cite{lisca}  implies that for $p$ and $q$ relatively prime, the lens space $L(pq,1)$ is not $\qq$--homology cobordant to any connected sum $L(p,a) \# L(q,b)$.  A simple consequence of the work here is that $L(pq,1)$ is not $\qq$--homology cobordant to any connected sum $M_p \#M_q$ where $H_1(M_p) = \zz_p$ and $H_1(M_q) = \zz_q$.

We let $\Theta^3_R$ denote the $R$--homology cobordism group of three-dimensional $R$--homology spheres.  Note that $\Theta^3_{\zz[1/p]}$ is generated by three-manifolds $M$ with $H_1(M)$ $p$--torsion.  There is a canonical map $$\Phi \co  \oplus_{p \in \calp} \Theta^3_{\zz[1/p]} \to \Theta^3_{\qq}.$$
Rochlin's Theorem and Furuta's result imply that the kernel of $\Phi$ is infinitely generated.  Our main result is the following:\vskip.05in

\noindent{\bf Proposition.} {\sl The cokernel of $\Phi$,   $  \Theta^3_{\qq}/ \Phi(\oplus_{p \in \calp} \Theta^3_{\zz[1/p]}) $,  contains an infinite free subgroup generated by lens spaces of the form $L(pq,1)$ and infinite two-torsion, generated by lens spaces of the form $L(4n^2+1, 2n)$.  An infinite subgroup is also generated by three-manifolds that bound $\qq$--homology balls topologically. } \vskip.05in

 We also present applications to the study of knot concordance and present  families of elements in the kernel $\Theta^3_\qq / \Theta^3_\zz \to \widehat{\Theta}^3_\qq$, where  $\widehat{\Theta}^3_\qq$ denotes the topological cobordism group.  Similar examples were presented in~\cite{hlr}, with the additional condition that bordisms were assumed to be Spin.

An important perspective is provided by considering the torsion linking form of three-manifolds, which yields a homomorphism $\Theta^3_\qq \to W(\qq /\zz)$, the Witt group of nonsingular $\qq/\zz$--valued linking forms on finite abelian groups.  According to~\cite{kk}  this homomorphism is surjective.  Again by Rochlin's theorem and Furuta's result, it has infinitely generated kernel (in the topological category it is conjecturally an isomorphism).  A basic result of Witt theory is that $W(\qq/\zz)$ splits into primary components, $\oplus_{p \in \calp} W(\ff_p) \xrightarrow{\cong} W(\qq/\zz)$, where $W(\ff_p)$ is the Witt group of linking forms of $\ff_p$--vector spaces and $\calp$ is the set of prime integers.  The conjecture that topological cobordism is determined by the linking form implies that $\widehat{\Theta}^3_\qq $ has a corresponding primary decomposition.  One thrust of our work here is to display the extent of the failure of the existence of  such a primary decomposition in the smooth setting.

The following commutative diagram organizes the groups of interest.  In the diagram, hats denote the topological category and $\calk$ denotes the kernel of the canonical homomorphism from the smooth to the topological $\qq$--homology cobordism group.  With the exception of the inclusion of the kernel, all horizontal arrows are surjective. Conjecturally, the right square consists of isomorphisms.
\vskip.1in

\begin{center}
$\begin{diagram}
\dgARROWLENGTH=1.4em
\node{  }  \node{ \oplus_{p \in \calp} \Theta^3_{\zz[1/p]}} \arrow{e}\arrow{s,r}{\Phi}\node{ \oplus_{p \in \calp} \widehat{\Theta}^3_{\zz[1/p]}}\arrow{e}\arrow{s,r}{\widehat{\Phi}}\node{ \oplus_{p \in \calp} W(\ff_p)}\arrow{s,r}{\cong} \\
\node{\calk}\arrow{e}   \node{ \Theta^3_{\qq}}\arrow{e}  \node{ \widehat{\Theta}^3_{\qq}}\arrow{e} \node{ W(\qq/\zz)}
\end{diagram}
$
\end{center}
\vskip.1in

\noindent  The proposition above states that $\Theta^3_\qq / \text{Image} (\Phi)$ is infinitely generated  containing an infinite free subgroup and infinite two-torsion and that furthermore,  the image of $\calk$ in  $\Theta^3_\qq / \text{Image} (\Phi)$ similarly contains an infinite  subgroup.

\vskip.05in

\noindent {\bf Definition.}  A three manifold $M$ is said to {\em split} if it represents a  class in the image of $\Phi$.  That is, a manifold does not split if it is nontrivial in the cokernel of $\Phi$.
\vskip.05in
\noindent{\bf Outline}  In Sections~\ref{secdefinitions}, \ref{secmetabolizers} and \ref{spincsection}  we present some of the basic definitions used throughout the paper, isolate a basic result concerning metabolizers of linking forms, and discuss \Spc--structures.  
Section~\ref{secobstructions}  presents one of our main results, describing an obstruction based on Heegaard--Floer $d$--invariants  to a class in $\Theta^3_\qq$ being in the image of $\oplus_{p \in \calp} \Theta^3_{\zz[1/p]}$.\vskip.05in
Following this we provide a series of examples:
\begin{itemize}
\item  Section~\ref{basiclensspace} demonstrates that lens spaces $L(pq, 1)$  with $p$ and $q$ square free and relatively prime do not split, and  extends this to finite connected  sums  of such  lens spaces, with all $p$ and $q$ distinct, thus proving that $  \Theta^3_{\qq}/ \Phi(\oplus_{p \in \calp} \Theta^3_{\zz[1/p]}) $ is infinite. Section~\ref{infiniteorderlens} further  extends this, demonstrating that the set of lens spaces of the form $L(pq,1)$ (with $p$ and $q$ now required to be prime) generate an infinite free subgroup of infinite rank contained in $  \Theta^3_{\qq}/ \Phi(\oplus_{p \in \calp} \Theta^3_{\zz[1/p]})  $.  \vskip.05in 

\item Section~\ref{order2lensspacesec} considers specific lens spaces of the form $L(4n^2 +1, 2n)$ to provide elements of order 2 in  $\Theta^3_\qq$ that do not split, in particular showing that $  \Theta^3_{\qq}/ \Phi(\oplus_{p \in \calp} \Theta^3_{\zz[1/p]}) $ contains 2--torsion.  Section~\ref{inftwotor} expands on this, providing an infinite family of independent elements of order 2.\vskip.05in

\item  Section~\ref{secbasicexample}  begins the examination of the failure of splittings among manifolds that do split topologically;  that is, we consider manifolds representing classes in  $\calk$.  The main example  is built from surgery on the  connected sum of the torus knot $T_{3,5}$ and the untwisted Whitehead double of the trefoil knot, $Wh(T_{2,3}) = D$. We show that $S^3_{15}(T_{3,5} \cs D)$ splits topologically  but not smoothly.  Section~\ref{secexamplestopsplit1}  generalizes that example to an infinite family, using $(p, p+2)$ torus knots, with $p$ odd. \vskip.05in 

\item Section~\ref{concordance} applies the results of Section~\ref{basiclensspace} to demonstrate the failure of a splitting theorem for knot concordance which, by a result of Stoltzfus~\cite{stoltzfus}, applies algebraically and in dimensions greater than 3.\vskip.05in

\item According to the Freedman's work~\cite{fr, freedman-quinn}, all homology spheres bound contractible 4--manifolds topologically, so $\Theta^3_\zz \subset \calk$.  In Section~\ref{toptrivialbordism} we outline the proof  that the quotient $\calk / \Theta^3_\zz$ contains an infinitely generated free subgroup.  This was proved in~\cite{hlr} with the added constraint that one restricts the cobordism groups by considering only manifolds that are $\zz_2$--homology spheres  or by requiring that all spaces have Spin--structures.  We briefly indicate how results here permit  one to remove those restrictions in the argument in~\cite{hlr}.

\end{itemize}
\vskip.05in

\noindent{\it Acknowledgements.}  We are grateful for Matt Hedden's help in better understanding   Heegaard--Floer homology.   His results regarding the Heegaard--Floer theory of doubled knots is central here, and our specific  examples  are inspired by those that Matt pointed us toward in our  collaborations with him.   
\section{Definitions}\label{secdefinitions}
We will consider $\qq$--homology   3--spheres: these are closed 3--manifolds $M^3$ with $H_n(M^3, \qq) \cong H_n(S^3, \qq)$ for all $n$.  For each such $M$ there is a symmetric linking form $\beta \co H_1(M) \times H_1(M) \to \qq/\zz $ which is nonsingular in the sense that the induced map $\beta^* \co H_1(M) \to \text{Hom}(H_1(M), \qq/\zz)$ is an isomorphism.   If $M = \partial X^4$ where $X$ is a compact 4--manifold and  $H_n(X, \qq) = H_n(B^4, \qq)$ for all $n$, then the kernel $\calm$ of the map $H_1(M) \to H_1(X)$ is a metabolizer for $\beta$ (see~\cite{CG1}). That is, $\calm^\perp = \calm$, and in particular  $|\calm|^2 = |H_1(M)|$. The Witt group $W(\qq/\zz)$ is built from  the set of all pairs $(G, \beta)$ where $G$ is a finite abelian group and $\beta$ is a non-degenerate symmetric bilinear form taking values in $\qq/\zz$.  There is an equivalence relation on this set:  $(G, \beta)\sim  (G', \beta')$ if $(G \oplus G', \beta \oplus -\beta')$ has a metabolizer, and under this relation it becomes an abelian group under direct sum, denoted $W(\qq/\zz)$.  It can be proved (e.g.~\cite{ahv}) that a pair $(G, \beta)$ is Witt trivial if and only if it has a metabolizer.  The proof of this fact includes the following, which we will be using.
\begin{prop}\label{metasplitprop}  If $(G_1, \beta_1) \oplus (G_2, \beta_2)$ has metabolizer $\calm$ and $  (G_2, \beta_2)$ has metabolizer $\calm_2$, then 
$\calm_1 = \{ g \in G_1\ |\  (g,h) \in \calm \text{ for some } h \in \calm_2 \}$ is a metabolizer for $(G_1, \beta_1)$.
\end{prop}
The Witt groups $W(\qq / \zz, \left<p\right>)$ are defined as is $W(\qq / \zz)$, considering only $p$--torsion abelian groups, and the decomposition   $W(\qq/\zz)\cong \oplus_{p \in \calp }  W(\qq / \zz, \left< p \right>)  $ is easily proved.  The Witt group of non-degenerate symmetric forms on $\ff_p$--vector spaces is denoted $W(\ff_p)$. The inclusion  $W(\ff_p) \to  W(\qq / \zz, \left<p\right>)$ is an isomorphism.  In the proof of this, the inclusion is clearly injective, and an inverse map  $   W(\qq / \zz, \left<p\right>) \to W(\ff_p) $ is explicitly constructed via ``divessage''~\cite{ahv, mh}.

Let $R$ be a commutative ring.  Two closed 3--manifolds, $M_1$ and $M_2$, are called \emph{$R$--homology cobordant} if there is a compact smooth  4--manifold $X$ with boundary the disjoint union   $ M_1 \cup -M_2$ such that  the inclusions $H_*(M_i, R) \to H_*(X, R)$ are isomorphisms.  Equivalently, they are $R$--cobordant, written $M_1 \sim_R  M_2$, if $M_1 \cs - M_2$ bounds an $R$--homology 4--ball.  The set of $R$--cobordism classes of $R$--homology spheres  forms an abelian  group with operation induced  by connected sum.  This group is denoted  $\Theta^3_R$.

The ring $\zz[1/p]$ is the ring of integers with $p$ inverted, consisting of all rational numbers with denominators a power of   $p$.  A closed 
\hbox{3--manifold} $M$ is a $\zz[1/p]$--homology sphere if and only if $H_1(M)$ is $p$--torsion.
The linking form provides well-defined homomorphisms $\Theta^3_\qq \to  W(\qq / \zz)$ and $\Theta^3_{\zz[1/p]} \to  W(\ff_p)$ for which the following diagram commutes.

\begin{center}
$\begin{diagram}
\dgARROWLENGTH=1.5em
\node{ \oplus_{p \in \calp}  {\Theta}^3_{\zz[1/p]}} \arrow{e,t}{ } \arrow{s,r}{\Phi}\node{ \oplus_{p \in \calp} W(\ff_p)}\arrow{s,r}{\cong} \\
\node{  {\Theta}^3_{\qq}} \arrow{e,t}{}\node{ W(\qq/\zz)}\\
\end{diagram}
$
\vskip-.3in
$\ $
\end{center}
If we switch to the topological category, all these maps are conjecturally isomorphisms.
\section{Metabolizers for connected sums}\label{secmetabolizers}          
\subsection{Metabolizers}
If a connected sum of 3--manifolds bounds a rational homology ball, the associated metabolizer of the linking form does not necessarily split relative to the connected sum.  However, the existence of the connected sum decomposition does place constraints on the metabolizer.

\begin{theorem}\label{metabolizerthm2}  If $p$ is prime, $G$ is a finite abelian group, and a given  nonsingular linking form $\beta_1 \oplus \beta_2$ on $\zz_p \oplus G$ has metabolizer $\calm$, then for some $a \in G$, $(1,a) \in \calm$.
\end{theorem}

\begin{proof} Let $G_p$ denote the $p$--torsion in $G$.  There is a metabolizer   $\calm_p$ for the form restricted to $\zz_p \oplus G_p$.  If $\calm_p \subset G_p$, then it would represent a metabolizer for the linking form restricted to $G_p$, implying that the order of $G_p$ is an even power of $p$.  But since the form on    $\zz_p \oplus G_p$ is metabolic, the order of $G_p$ must be an odd power of $p$.  
It follows that there is an element $(a',a'') \in \calm_p$ with $a' \ne 0$.  Multiplying by $(a')^{-1} \mod p$, we see that $ (1,a) \in \calm_p \subset \calm$ for some $a \in G_p$.  
\end{proof}

In the following corollary, for each  integer $k$, $G_k$ denotes a finite abelian group  of order dividing a  power of $k$. 
\begin{corollary}\label{productthm} If $m$ is a square free integer, $G_m\oplus G_n$ is a finite abelian group with $\gcd(m,n)=1$,
and a given linking form  $\beta_1 \oplus \beta_2\oplus \beta_3$ on $\zz_m \oplus G_m\oplus G_n$ has metabolizer $\calm$, then for some $a \in G_m$, $(1,a,0) \in \calm$.
\end{corollary}
  
\begin{proof}Write $\zz_m = \zz_{p_1} \oplus \cdots \oplus \zz_{p_k}$. By Theorem~\ref{metabolizerthm2}, the projection of $\calm$ to each $\zz_{p_i}$ summand is surjective.  Since the $p_i$ are relatively prime, the projection to $\zz_m$ is similarly surjective.
\end{proof}
In order to construct elements of infinite order, we will need to consider multiples of linking forms.  Without loss of generality, we will be able to assume that the multiplicative factors are divisible by 4.
 
\begin{theorem}\label{thmmultiplemetab} Suppose that $p$ is prime and the nonsingular form $4 k(\beta_1 \oplus \beta_2)$ on $(\zz_p \oplus G)^{4k}$ has a metabolizer $\calm$.  Then   $\calm$    contains an element of the form $(1, 1, \ldots , 1, \alpha_{2k+1}, \cdots , \alpha_{4k}) \oplus b$ for some set of $\alpha_i \in \zz_p$ and some $b \in G^{4k}$.
\end{theorem}

\begin{proof}  The Witt group $W(\qq/\zz)$ is 4--torsion~\cite{mh}, and thus $4 k \beta_2$ has a metabolizer $\calm'$.  By Proposition~\ref{metasplitprop}, the set of elements $x$ such that $(x,y) \in \calm$ for some $y \in \calm'$ is a metabolizer, denoted $\caln$, for $4k \beta_1$, and thus is $2k$--dimensional.  As argued in~\cite{LN1}, a simple application of the Gauss--Jordan algorithm applied to a generating set for $\caln$  yields  a generating set consisting of vectors of the form $(1,0 ,0 ,0 , \ldots , 0, *, * \ldots )$, $(0,1 ,0 ,0 , \ldots , 0, *, * \ldots )$, $(0,0, 1 ,0 , \ldots , 0, *, * \ldots )$, $\ldots$, where each initial sequence of a 1 and 0s is of length $2k$.  

By adding these vectors together, we find that the metabolizer  $\caln$ contains an element of the form $ (1,1,\cdots, 1, \alpha_{2k+1} , \cdots, \alpha_{4k})  \in \zz_p^{4k}$.  Finally, since each element in $\caln$ pairs with an element in the metabolizer $\calm'$ to give an element in $\calm$, we get the desired element $b$.
\end{proof}

\section{\Spc--structures}\label{spincsection}  We need the following facts about    \Spc($Y)$, the set of  \Spc--structures on an arbitrary space $Y$.
\begin{itemize}
\item The first Chern class is a map $c_1 \co   \text{\Spc}(  Y )  \to H^2(Y).$\vskip.05in

\item  There is a transitive action $H^2(Y) \times \text{\Spc}(Y) \to \text{\Spc}(Y) $ denoted $(\alpha,\spinc) \to \alpha \cdot \spinc$.\vskip.05in

\item  For $Y \subset W$, the restriction map $r$  is functorial:  If $\spinc \in  \text{\Spc}(W)$,  $\alpha \in H^2(W)$   then
$$r(\alpha  \cdot \spinc) = r(\alpha) \cdot r(\spinc).$$\vskip.05in

\item  For all $\alpha \in H^2(Y) $ and $\spinc \in \text{\Spc}(Y)$, $c_1(\alpha \cdot \spinc) - c_1(\spinc) = 2 \alpha $.\vskip.05in

\item  As a corollary, if $|H^2(Y)|$ is finite and odd, then $c_1 \co  \text{\Spc}(  Y )  \to H^2(Y) $ is a bijection. \vskip.05in

\item   There is a canonical bijection:  \Spc($Y \cs W) \to$ \Spc($Y)  \times$ \Spc($W$).   \vskip.05in
\end{itemize}

For every smooth 4--manifold $X$, the set \Spc$(X)$ is nonempty.  (See~\cite{gs} for a proof.)  As a consequence, we have the following.

\begin{theorem}\label{extendthm} Let $ N  = \partial X$  and  let $  \spinc \in \text{\Spc}(N) $ be the restriction of a \Spc--structure on $X$.  Then the set of  \Spc--structures on $N$ which extends to $X$ are those of the form  $ \alpha  \cdot \spinc $ for   $ \alpha $ in the image of the restriction map $r\co H^2(X) \to H^2(N)$.
\end{theorem}

\subsection{Identifying $H_1(N)$ and $H^2(N)$.}  Suppose that  $N$ is a rational homology 3--sphere bounding a rational homology ball $X$. Then by Poincar\'e duality,  $H_1(N) \cong H^2(N)$.  We have denoted kernel($H_1(N) \to H_1(X))$ by $\calm$.  Via duality, it corresponds to the image of $H^2(X)$ in $H^2(N)$.  Thus, we will use $\calm$ to denote this subgroup of $H^2(N)$.
\subsection{Spin--structures}  If the order $|H_1(M)|$ is odd, then there is a unique Spin--structure on $M$  that lifts to a canonical \Spc--structure that we will denote $\spinc_0$.  With this, there is a natural identification of $H^2(M)$ with \Spc$(M)$.  However, we face the complication that in assuming that $M$ bounds a rational homology 4--ball $X$, we cannot assume that $X$ has a Spin--structure.  The following result permits us to adapt to this possibility.  (In addition to playing a role in considering splittings of classes in $\Theta^3_\qq$, in Section~\ref{toptrivialbordism} we will use this result to extend a theorem from~\cite{hlr}  in which an added hypothesis was needed to ensure the existence of a Spin--structure on $X$.)

\begin{theorem}\label{lemmaextendspin} Suppose that $ N_1 \cs N_2 = \partial X$ for some smooth rational homology 4--ball $X$ and that the order of $H_1(N_1)$ is odd.  Then the image of the restriction map \Spc$(X) \to $ \Spc$(N_1)$ contains the Spin--structure $\spinc_0 \in $ \Spc$(N_1)$.  In particular, every element in the image of this restriction map is of the form $\alpha \cdot \spinc_0$ for   $\alpha \in \text{Image}(H^2(X) \to H^2(N_1))$. 
\end{theorem}

\begin{proof}  Let $H =  \text{Image}(H^2(X) \to H^2(N_1))$ and  $S =\text{Image}($\Spc$(X) \to$ \Spc$(N_1))$. As usual, the choice of an element $\spinc \in S$ determines a bijection between $H$ and $S$.  In particular, the number of elements in $S$ is the same as in $H$, which is odd.  Conjugation defines an involution on $S$ which commutes with restriction.  Thus, since $S$ is odd, conjugation has a fixed point in $S$.  But the only fixed element under conjugation is the Spin--structure, since $ c_1(\bar{\spinc}) = -c_1(\spinc)$.
\end{proof}

\section{Basic obstructions from $d$--invariants}\label{secobstructions}  
To each rational homology 3--sphere $M$ and $\spinc \in \text{\Spc}(M)$ there is associated an invariant $d(M,\spinc) \in \qq$, defined in~\cite{os2}.  It is additive under connected sum: $d(M\cs N, (\spinc_1 , \spinc_2)) = d(M\ , \spinc_1  ) + d(  N,   \spinc_2) $. A key result relating the $d$--invariant and bordism is the following from~\cite{os2}.

\begin{theorem} 
If $M = \partial X$ with $H_*(X,\qq) \cong H_*(B^4,\qq)$,  and $ \spinct  \in \text{\Spc}(X)$, then $d(M,\spinct |_{M}) = 0$.\end{theorem}

\subsection{Obstruction theorem}
Suppose that $|H_1(M)|$  is odd and $\spinc_0$ is the unique  Spin--structure on $M$.   For $\alpha \in H^2(M)$, we abbreviate $d(M, \alpha \cdot \spinc_0) $ by $d(M, \alpha)$. 

\begin{definition} $\bar{d}(M, \alpha) = d(M, \alpha) -d(M,0)$.
\end{definition}

The following result will be sufficient to prove that $\Theta^3_\qq / \Phi(\oplus \Theta^3_{\zz[1/p]})$ is infinite.
\begin{theorem}\label{obstructthm}Suppose $\{M_i\}$ is a collection of 3--manifolds for which $H_1(M_i) = \zz_{m_i} \oplus \zz_{n_i}$, where $m_i$ and $n_i$ are square free and odd, and the full set $\{m_i, n_i\}$ is pairwise relatively prime.   If a finite connected sum $\cs_{k=1}^N \pm M_{i_k}$ represents a class in $\Theta^3_{\qq}$ that is in the image of $\oplus_{p} \Theta^3({\zz[1/p]})$, then for all $i=i_k, 1 \le k \le N$, and for all $(a,b) \in \zz_{m_i} \oplus \zz_{n_i}$, $$\bar{ d} (M_{i}, (a,b) ) =    \bar{d}(M_{i}, (a,0) ) +  \bar{d}(M_{i},  (0,b)) .$$
\end{theorem}
\begin{proof}Suppose that $Y =  \cs_k \pm M_{i_k} \in \Phi(\oplus_{p} \Theta^3({\zz[1/p]}))$. We consider $k= 1$,  abbreviating  $M_{i_1} = M$ and $H_1(M) \cong \zz_m \oplus \zz_n$.
Suppose that $Y$ is in the image.  Then $Y \cs \oplus Y_{p_i} = \partial X$ for some collection of $Y_{p_i}$ which are $\zz[p_{i}^{-1}]$--homology spheres and $X$ is a rational homology ball.    Collecting summands, we can write $M  \cs N_m \cs N_n \cs N = \partial X$, where the prime factors of $|H_1(N_m)|$ all divide $m$, the prime factors of  $|H_1(N_n)| $ all divide $n$, and $|H_1(N)|$ is relatively prime to $mn$.
Let $(\spinc_0, \spinc_1, \spinc_2, \spinc_*) \in \text{Image(\Spc}(X))$.   (By Theorem~\ref{lemmaextendspin}  we can assume that the structure $\spinc_0 \in \text{\Spc}(M)$ is the Spin--structure.)  Then by Corollary \ref{productthm}, for all $a \in \zz_m$ and $b \in \zz_n$, there are elements $a' \in H_1(N_m)$ and $b' \in H_1(N_n)$ such that:  \vskip.05in 

 \begin{itemize}
\item $((a,0) \cdot \spinc_0, a' \cdot \spinc_1, \spinc_2, \spinc_*) \in \text{Image(\Spc}(X))$. \vskip.05in

\item $(  (0,b) \cdot \spinc_0,   \spinc_1,b' \cdot  \spinc_2, \spinc_*) \in \text{Image(\Spc}(X))$. \vskip.05in

\item $(  (a,b) \cdot \spinc_0,  a' \cdot  \spinc_1,b' \cdot  \spinc_2, \spinc_*) \in \text{Image(\Spc}(X))$. \vskip.05in
\end{itemize}
Thus, we have the following vanishing conditions on the $d$--invariants:
\vskip.05in
\begin{itemize}
\item $d(M  , \spinc_0) + d(N_m, \spinc_1) + d(N_n, \spinc_2)  + d(N, \spinc_*)   = 0$.\vskip.05in

\item $d(M  , (a,0)\cdot \spinc_0) + d(N_m, a' \cdot \spinc_1) + d(N_n, \spinc_2)  + d(N, \spinc_*)   = 0$.\vskip.05in

\item $d(M  ,(0,b) \cdot  \spinc_0) + d(N_m, \spinc_1) + d(N_n, b' \cdot \spinc_2)  + d(N, \spinc_*)   = 0$.\vskip.05in

\item $d(M  , (a,b) \cdot  \spinc_0) + d(N_m, a' \cdot \spinc_1) + d(N_n,b' \cdot  \spinc_2)  + d(N, \spinc_*)   = 0$.\vskip.05in
\end{itemize}

Subtracting the second and third equality from the sum of the first and  fourth yields:
$$d(M , (a,b) \cdot \spinc_0) -    d(M , (a,0)\cdot  \spinc_0) -  d(M , (0,b) \cdot\spinc_0) +d(M ,  \spinc_0) = 0.$$
Recalling that $\bar{d}(M, \alpha )$ denotes $d(M, \alpha \cdot \spinc_0) - d(M, \spinc_0)$, this can be rewritten as $$\bar{d}(M, (a,b)) - \bar{d}(M, (a,0 ))- \bar{d}(M, (0,b)) =0.$$ Repeating for each $M_i$ completes the proof of the theorem.
\end{proof}


\section{Lens Space Examples:  $L(pq,1)$.}\label{basiclensspace}
Let $\{p_i,q_i\}$ be a set of pairs of odd integers such that the union of all pairs are pairwise relatively prime.  We prove:

\begin{theorem}\label{lensspacethm}  No 
finite linear combination $\cs_k \pm L(p_{i_k}q_{i_k} ,1)$   represents an element in the image $\Phi(\oplus_{p \in \calp} \Theta^3_{\zz[1/p]}) \subset \Theta^3_\qq$.  
\end{theorem}

\begin{proof} We consider the first term $L(p_1q_1,1)$ and simplify notation by writing $p = p_1$ and $q = q_1$.
By Theorem~\ref{obstructthm} we would have  for all $(a,b) \in \zz_{p} \oplus \zz_{q}$, $$\bar{ d} (L(pq,1), (a,b) ) =    \bar{d}(L(pq,1), (a,0) ) +  \bar{d}(L(pq,1),  (0,b)) .$$
According to~\cite{os2}, for some enumeration of  \Spc--structures on $L(m,n)$, denoted $\spinc_i$, $0\le i <m$, if we let $D(m,n,i) = d(-L(m,n), \spinc_i)$, there is the recursive formula:
$$D(m,n,i) = \frac{ mn - (2i +1 - m-n)^2}{4mn} - D(n,m',i'),$$
where the primes denote reductions modulo $n$, $0<n <m$, and   $0 \le i <  m$.  The base case in the recursion is by definition $D(1,0,0) = 0$.
For every \Spc--structure $\spinc$ there is a conjugate structure $\bar{\spinc}$ for which $d(M, \spinc) = d(M, \bar{\spinc})$ and $\spinc \ne \bar{\spinc}$ unless $\spinc$ is the Spin--structure.  We claim that for $L(pq,1)$, the \Spc--structure $\spinc_0$ does correspond to the Spin--structure.  To see this, observe that an algebraic compuation shows   $4pqD(pq,1,i) = -4i^2 +4pq i +pq(1-pq)$ and in particular, $pqD(pq,1,0) = pq(1-pq)$.  The difference $4pqD(pq,1,i) - 4pqD(pq,1,0) = 4i(pq -i)$,   does not take on the value 0 for any $0<i<pq$.  Since the value of $D(pq,1,0)$ is unique among the $d$--invariants, it must correspond to the Spin--structure.
In applying Theorem~\ref{obstructthm}, we identify $\zz_p \oplus \zz_q \cong \zz_{pq}$, so that the pair $(a,b) \in \zz_p \oplus \zz_q  $ corresponds  to $aq+bp \in \zz_{pq}$. In this case, the  criteria becomes 
$$D(pq,1, ap+bq)  -   D(pq,1,ap) - D(pq,1,bq)+   D(pq,1,0)  = 0.$$
Certainly $p+q <pq$, so we can apply the formula for $D$ with $a = b =1$ .  However, in this case the sum is immediately calculated to equal $-2 \ne 0$. 
\end{proof}


\section{Infinite order examples}\label{infiniteorderlens}
The examples of the previous section are sufficient to demonstrate  that the quotient $\Theta^3_\qq /  \Phi(\oplus \Theta^3_{\zz[1/p]})$ is infinite.  We now present an argument to show it contains an infinite free subgroup.  To carry out this argument we need to make the additional assumption of primeness for the relevant $p$ and $q$.   Let $\{p_i , q_i\}$ be a set of distinct odd prime pairs with all elements distinct.  This section is devoted to the proof of the following theorem.

\begin{theorem}\label{infiniteorderthm} The lens spaces $L(p_i q_i, 1)$ are linearly independent  in the quotient $\Theta^3_\qq / \Phi(\oplus \Theta^3_{\zz[1/p]})$.
\end{theorem}

\subsection{Notation}  
Suppose that  $\sum_i b_i L(p_iq_i ,1) \subset \text{Image\ } (\Phi)$.  We can assume that $b_1 \ne 0$.  We simplify notation, writing $p $ and $q $ for $p_1$ and $q_1$, respectively.
There is no loss of generality in assuming that for all $i$,  $b_i = 4k_i$ for some $k_i$, and write $k= k_1$.  At times we also abbreviate $L(pq,1) = L_{pq}$.

Following our earlier approach,   we will show that a contradiction arises from the assumption that   $N = 4k L(pq,1)  \cs M_p \cs M_q \cs M_*= \partial X$ for some rational homology 4--ball $X$, where the orders of $H_1(M_p)$ and $H_1(M_q)$ are powers of $p$ and $q$, respectively, and the order of $H_1(M_*)$ is relatively prime to $pq$.

According to Theorem~\ref{thmmultiplemetab}, the $p$--primary part of the associated metabolizer, $\calm_p$, includes a vector $A  = ((1, \ldots, 1, \alpha_{2k+1} , \ldots , \alpha_{4k}), g) \in  (\zz_p)^{4k} \oplus H_1(M_p)$.  Similarly, the $q$--primary part of the associated metabolizer, $\calm_q$, includes a vector $B = ((1, \ldots, 1, \beta_{2k+1} , \ldots , \beta_{4k}), h) \in  (\zz_q)^{4k} \oplus H_1(M_q)$.  
\subsection{ Constraints on the $d$--invariants}
We let the Spin--structures on $L(pq,1)$, $M_p$, and $M_q$ be $\spinc_0, \spinc_0'$ and $\spinc_0''$, respectively.    Consider now the vectors 0,  $aA$, $bB$, and $aA + bB \in \calm$. Computing the $d$-invariant associated to each, we find that each of the following sums is 0.

\begin{itemize}
\item $2k d(L_{pq}, s_0) + \sum_{i = 2k+1}^{ 4k}  d(L_{pq} ,  \spinc_0) + d(M_p, \spinc_0' ) + d(M_q, \spinc_0'') +d(M_*,\spinct) $.\vskip.05in

\item $2k d(L_{pq}, aq \cdot s_0) + \sum_{i = 2k+1}^{4k}  d(L_{pq},  aq \alpha_i \cdot \spinc_0) + d(M_p, a g \cdot \spinc_0' ) + d(M_q,  \spinc_0'') +d(M_*,\spinct)$.\vskip.05in

\item $2k d(L_{pq}, bp\cdot s_0) + \sum_{i = 2k+1}^{ 4k}  d(L_{pq},  bp \beta_i \cdot \spinc_0) + d(M_p,   \spinc_0' ) + d(M_q,  bh\cdot \spinc_0'') +d(M_*,\spinct)$.\vskip.05in

\item $2k d(L_{pq}, (aq + bp) \cdot s_0) + \sum_{i = 2k+1}^{ 4k}  d(L_{pq} , (aq \alpha_i  + bp \beta_i) \cdot \spinc_0) +  d(M_p, a g \cdot \spinc_0') )+ d(M_q,  bh\cdot \spinc_0'')+d(M_*,\spinct) $.\vskip.05in

\end{itemize}

\noindent{\bf Note.} We have again used  that the inclusion   $\zz_p \subset \zz_{pq}$ takes $\alpha$ to $\alpha q$, and similarly for $\zz_q$ and $\beta$.
We now take the sum of the first and last equation, and subtract the sum of the middle two.    The result is that for some set of $a_i$ and $b_i$:
$$  \hskip-.5in 2k\left( { d} (L_{pq}, aq + bp) -   { d} (L_{pq}, aq )-   { d} (L_{pq},   bp)   +  {d}  (L_{pq}, 0)  \right)+$$  
$$\hskip.5in  \sum_{i = 2k+1}^{4k}  \left(   { d}  (L_{pq}, a_i q + b_i p) -  { d}  (L_{pq}, a_i q )-   { d} (L_{pq},   b_i p) +   { d}   (L_{pq},   0)\right)
= 0.$$
We now introduce further notation: let 
$$ \delta  (L_{pq} , a,b) =  {d} (L_{pq}, aq + bp) - {d} (L_{pq}, aq)-   {d} (L_{pq},  bp) +  {d} (L_{pq}, 0) .$$  With this, we have proved the following lemma.

\begin{lemma}\label{deltaboundlemma} If  the lens spaces $L_{p_i q_i}$ are linearly  dependent in $\Theta^3_\qq / \Phi(\oplus \Theta^3_{p})$ and, for $p=p_1$ and $q=q_1$, $L_{pq}$ has nonzero coefficient in some linear relation,  then for  all $a$ and $b$ there are $k$, $a_i$ and $b_i$ such that,
$$2k  {\delta} (L_{pq},a,b) + \sum_{i = 2k+1}^{ 4k}  {\delta} (L_{pq},a_i ,b_i ) = 0. $$
\end{lemma} 

\subsection{Computation of bounds on $\delta (L_{pq},a,b)$}
Note that $\delta(L_{pq},a,b)=0$ if $a=0$ or $b=0$. Given Lemma~\ref{deltaboundlemma}, the proof of Theorem~\ref {infiniteorderthm} is completed with the following result.

\begin{lemma}  For all $a \ne  0 \mod p$ and $b \ne 0 \mod q$, $\delta( L_{pq}, a,b)  <  0 $. \end{lemma}
\begin{proof}  All \Spc--structures are included by considering the range $-\frac{p-1}{2} \le a \le \frac{p-1}{2}$ and $-\frac{q-1}{2} \le b \le \frac{q-1}{2}$.  By symmetry we can exclude the case $a<0$.  Since the formula for the $d$--invariant $d(L(pq,1), i)$ assumes $i \ge 0$, there are three cases to consider.

\begin{enumerate}
\item $a > 0 , b>0$.\vskip.05in 

\item $a>0 , -\frac{aq}{p} < b <0$. \vskip.05in 

\item $a>0 , b< -\frac{aq}{p}$. \vskip.05in
\end{enumerate}
The formula for the $d$--invariant  in the current case is 
$$4 n(d(L(n,1), i)) =   { n- (2i + 1 -n -1)^2}  = n - n^2 + 4n i - 4i^2,$$  for $0 \le i <  n$.  
We now compute $4pq \delta(L_{pq}, aq+bp)$ in each of the three cases.  First note that $\delta(L_{pq},aq + bp) = d(L_{pq}, aq+bp) - d(L_{pq}, aq) - d(L_{pq},bp) + d(L_{pq},0).$  In places we write $pq = n$ to simplify the appearance of the formula.
\begin{enumerate}
\item  Since all entries are now positive we find
\begin{align*}
4n \delta(L_{pq},a,b)  &=    \left( n - n^2 + 4n(aq +bp) - 4(aq+bp)^2 \right)\\
&-  \left( n - n^2 + 4n(aq ) - 4(aq)^2 \right)\\
&-  \left( n - n^2 + 4n(bp ) - 4(bp)^2 \right)\\
&+   \left( n - n^2 + 4n(0 ) - 4(0)^2 \right).
\end{align*}
This simplifies to $-8abpq$, which is negative.
\vskip.05in

\item  In this case  $bp<0$, so we replace $d(L_{pq}, bp) $ with $d(L_{pq}, -bp)$ in the computation.  
\begin{align*}
4n \delta(L_{pq},a,b) & =    \left( n - n^2 + 4n(aq +bp) - 4(aq+bp)^2 \right)\\
&-  \left( n - n^2 + 4n(aq ) - 4(aq)^2 \right)\\
&-  \left( n - n^2 + 4n(-bp ) - 4(-bp)^2 \right)\\
&+   \left( n - n^2 + 4n(0 ) - 4(0)^2 \right).
\end{align*}
This simplifies to give $-8b(a-p)pq$.  Since $b<0$ and $a < \frac{p-1}{2}$, this is negative.\vskip.05in

\item In this case, both $bp$ and $aq + bp <0$.  Thus, we compute 
\begin{align*}
4n \delta( L_{pq},a,b ) & =    \left( n - n^2 + 4n(-aq -bp) - 4(-aq-bp)^2 \right)\\
&-  \left( n - n^2 + 4n(aq ) - 4(aq)^2 \right)\\
&-  \left( n - n^2 + 4n(-bp ) - 4(-bp)^2 \right)\\
&+   \left( n - n^2 + 4n(0 ) - 4(0)^2 \right).
\end{align*}
This simplifies to give $-8apq(b+q)$.  Since $b > -\frac{q-1}{2}$, this is again negative.
\end{enumerate}
\end{proof}

\section{An Order 2 lens space that does not split}\label{order2lensspacesec}
We now consider a  lens space that represents 2--torsion in $\Theta^3_\qq$.  Let  $M= L(65,8)$;  since $8^2 = -1 \mod 65$,   $M= -M$  and $2M = 0 \in \Theta^3_\qq$.  We show that $M$ does not split.
It follows quickly from the fact that $ L(65,8) $ is of finite order in $\Theta^3_\qq$ that for the Spin-structure $\spinc^*$, $d(L(65,8), \spinc^*) = 0$.  On can compute directly from the formula for $D$ given above that the value 0 is realized only by $\spinc_{36}$.     Thus, in applying Theorem~\ref{obstructthm} we identify the homology class $x \in H_1(L(65,8))$ with the \Spc--structure $\spinc_{36+x}$, where the index is taken modulo 65.  
The matrix in  Figure~\ref{fig1a} presents the values of $d(L(65,8), 13a + 5b)$ (multiplied by 65 to clear denominators).  Rows correspond to the values of $a$ and columns to $b$.   The central row and left column correspond to $a=0$ and $b=0$ respectively. Symmetry permits us to list only the values with $b \ge 0$.       In Figure~\ref{fig2a} we list the differences, $d(L(65,8), 13a + 5b) - d(L(65,8), 13a ) -d(L(65,8), 5b)$, with the nonzero entries demonstrating the failure of additivity.
\begin{figure}[h]
$$\begin{array}{c|c|c|c|c|c|c|c|} 
& b=0 & b= 1 & b= 2 & b= 3 & b= 4 & b=5 & b= 6\\
\hline
a=2        &  -52 & 18 & -32 & 58 &28 & 8 & 128 \\
\hline 
a=1     & 52 & -8 & 72 & 32 &2 &112 & -28\\
\hline 
a=0       &0 & 70 & 20 & -20 & 80 &-70& -80 \\
\hline
a=-1     &52 &-8 & -58 & 32 &-128 & 18 & -28  \\
\hline
a =-2         & -52 &-112 &-32 & -72 & 28 &8 & -2  \\
\hline
\end{array}$$
\caption{$65\,d(L(65,8), 13a + 5b) $}\label{fig1a}
\end{figure}

\begin{figure}[h]
$$\begin{array}{c|c|c|c|c|c|c|c|} 
& b=0 & b= 1 & b= 2 & b= 3 & b= 4& b=5  & b= 6\\
\hline
a=2        & 0 & 0 & 0 & 2 & 0 & 2 & 4 \\
\hline 
a=1     & 0 & -2 & 0 & 0 & -2 & 2 & 0 \\
\hline 
a=0       & 0 & 0 & 0 & 0 & 0 & 0 & 0 \\
\hline
a=-1     & 0 &-2 & -2 & 0& -4 &0 & 0 \\
\hline
a =-2         & 0 & -2 & 0 & 0 & 0 & 2 & 2 \\
\hline
\end{array}$$
\caption{$d(L(65,8), 13a + 5b) - d(L(65,8), 13a ) -d(L(65,8), 5b)$}\label{fig2a}
\end{figure}

\section{Infinite 2--torsion}\label{inftwotor}
We now generalize the previous example to describe an infinite subgroup of $\Theta^3_\qq$ consisting of 2--torsion that injects into the quotient      $\Theta^3_\qq / \Phi(\oplus_{p \in \calp} \Theta^3_{\zz[1/p]})$.
Consider the family $N_n = L( 4(5n +1)^2 +1, 2(5n+1))$; for $n=-1$ we have $-L(65,  8)$ as in the previous section, but we simplify the computations by restricting to $n>0$.  Expanding, we have
$N_n = L(5 ( 20n^2 +8n +1), 2(5n+1))$.  If $n \ne 3 \mod 5$, then $20n^2 + 8n +1$ is not divisible by 5.  By Appendix~\ref{appendpi} we can further assume that the $n$ are   selected so that $n$ is divisible by 5 and the set of integers $20n^2 + 8n +1$ are pairwise relatively prime and square free.   We enumerate the set of such $n$ as $n_i$ and abbreviate the corresponding  lens spaces   as $L(5p_i, q_i) = N_{n_i}$.
The remainder of this section is devoted to proving the following.

\begin{theorem}\label{twotorthm}The set   $\{N_{n_i}$\} generates an infinite subgroup consisting of elements of order 2 in $\Theta^3_\qq / \Phi(\oplus_{p \in \calp} \Theta^3_{\zz[1/p]})$.
\end{theorem}

To begin, we need to identify the Spin--structure.  We use the recursion formula
$$D(m,n,i) = \frac{ mn - (2i +1 - m-n)^2}{4mn} - D(n,m',i')$$
to compute relevant $d$--invariants.     
We are interested in the lens spaces  $L(4r^2 +1, 2r)$.  One step of the recursion reduces this to $L(2r, 1)$, and another step reduces it to $S^3$.  Since we need to reduce modulo $2r$, for $0\le i<4r^2+1$, let $y$ be the remainder of $i$ modulo $2r$ and $x$ the quotient so that $2rx+y=i$. So we write \Spc--structures as $\spinc_{2rx + y}$ for $0 \le y < 2r$ and  $0\le 2rx+y<4r^2+1$. Carrying out the arithmetic yields:

\begin{lemma}\label{spinlemma} For any $r>0$, $x$ and $y$ with $0\le y<2r $ and $0\le 2rx+y<4r^2+ 1$,
 \begin{enumerate}
\item  $d(L(4r^2 +1, 2r), \spinc_{2rx+y}) =\frac{2\left(rx^2+(y-r(2r+1))x-r(y^2-(2r-1)y-r)\right)} {4r^2+1}$.\vskip.05in

\item The discriminant of the numerator, viewed as a quadratic  polynomial  in the variable $x$, is $4(y-r)^2(4r^2+1)$. Moreover,  it is the square of an integer if and only if $y=r$.\vskip.05in
\item $d(L(4r^2 +1, 2r), \spinc_{2rx+y})=0$ if and only if $x=r$ and $y=r$.\vskip.05in
\item The Spin--structure on $L(4r^2 +1, 2r)$ is $\spinc_{2r^2 +r}$.\vskip.05in
\end{enumerate}
\end{lemma}

In our case   $r = 5n +1$ and the Spin--structure is $\spinc_{50n^2+25n+3}$.
\begin{proof}[Proof Theorem~\ref{twotorthm}]   For each $n$, we write  $N_n = L(5p_n , q_n)$ and assume that some linear combination $\sum N_{n_i} = 0 \in  \Theta^3_\qq / \Phi(\oplus_{p \in \calp} \Theta^3_{\zz[1/p]})$.  We write the first term in the sum as $N = L(5p,q)$ where $p =  20n^2 + 8n +1$.  

Since  the sum splits, for some collection of primes $r_j$ and manifolds $M_{r_j}$ with $H_1(M_{r_j}) $ being $r_j$--torsion, we have $$N \#_{i >1} N_{n_i} \#_j M_{r_j} = \partial X,$$ where $X$ is a rational homology ball.  We can collect terms as $N \# M_p \# M_m = \partial X$ where $M_p$ includes all the $M_{r_j}$ for which $r_j$ divides $p$, and $M_m$ contains all the other summands, including all the $N_{n_i}$ with $i>1$.

The homology of this connected sum of   three manifolds splits into the direct sum of three groups: $(\zz_5 \oplus \zz_p) \oplus G_p \oplus G_m$, where the order of $G_p$ is a product of prime factors of $p$, 5 does not divide the order of $G_p$, and the orders of $G_p$ and $G_m$ are relatively prime.  It follows that the $5$--torsion in the metabolizer, $\calm_5$, is contained in   $ (\zz_5,0)   \oplus   0   \oplus G_m$.  The direct sum of all primary parts of the metabolizer for primes that divide $p$, $\calm_p$, is  contained in $\calm_p =  (0,\zz_p )   \oplus   G_p   \oplus 0$.

Now, as in our previous arguments,  $\calm_5$ contains an element of the form $(1,0) \oplus 0 \oplus a''$ and $\calm_p$ contains an element $(0,1) \oplus b'' \oplus 0$.  Continuing as in the early proofs, we find that for all $a$ and $b$, 
$$\bar{d}(L(5p,q), (a,b)) = \bar{d}(L(5p,q), (a,0)) +  \bar{d}(L(5p,q), (0,b) ).$$  Or, writing $\zz_5 \oplus \zz_p$ as $\zz_{5p}$, 
$$\bar{d}(L(5p,q), pa + 5b) = \bar{d}(L(5p,q), pa ) +  \bar{d}(L(5p,q),5b ).$$
Since $L(5p,q)$ is of order two, for the Spin--structure the $d$--invariant vanishes, so the $\bar{d}$--invariant is the same as the $d$--invariant.  We let $a = 1$ and $b=-1$, and arrive at a contradiction by showing the following equality does not hold:
$$ {d}(L(5p,q), p-5) =  {d}(L(5p,q),  p ) +   {d}(L(5p,q), -5 ).$$
To apply Lemma~\ref{spinlemma} we need to express each of   $(50n^2+25n+3)+p-5$,  $(50n^2+25n+3)+p $, and   $(50n^2+25n+3)-5  $, as $2(5n+1)x +y$.  Simple algebra yields the following pairs $(x,y)$ for these three respective \Spc--structures:

\begin{itemize}
\item  $a=1, b= -1 \longrightarrow (x,y) = ( 7n+1, 9n-3)$.\vskip.05in

\item  $a=1, b= 0 \longrightarrow (x,y) = ( 7n+1, 9n+2)$.\vskip.05in

\item  $a=0, b= -1 \longrightarrow (x,y) = ( 5n+1, 5n-4)$.\vskip.05in
\end{itemize}

Finally, one uses these expressions to determine that for all $n$, $$ {d}(L(5p,q), p-5) -  {d}(L(5p,q),  p ) -   {d}(L(5p,q), -5 )= 4.$$
Since the difference is not zero, no splitting exists and the proof of Theorem~\ref{twotorthm} is complete.
\end{proof}           
\section{Topologically split examples}\label{secbasicexample}  

In this section, we apply Theorem~\ref{obstructthm} to find examples of manifolds that split topologically but not smoothly.   We begin by  carefully examining an example in which the splitting exists smoothly, focusing on the computation of the $d$--invariants, and next illustrate the modifications which do not change its topological cobordism class, but alter it smoothly.  The deepest aspect of the work is in the determination of the $d$--invariants.
In brief, the manifold we look at is 15--surgery on the $(3,5)$--torus knot, $T_{3,5}$, denoted $ S^3_{15}(T_{3,5}) $.  This is homeomorphic to the connected sum $L(3,5) \cs - L(5,3)$.  Next, letting $D $ denote the untwisted double of the trefoil knot ($D = Wh(T_{2,3})$), which is topologically slice, we consider   $ S^3_{15}(T_{3,5} \cs D ) $, and prove that it does not split in the cobordism group.    

In this section and the next, and also Appendix~\ref{torusknotpoly}, we develop properties of the Heegaard-Floer complex of specific torus knots as well as tensor products of certain of these complexes.  Related and more extensive computations appear in~\cite{hhn}. 

\subsection{ $\bf \bar{d} (S^3_{15}(T_{3,5}),i ))$} 
We now determine the doubly filtered   Heegaard-Floer complex   
$CFK^\infty(S^3, T_{3,5})$.  This complex is by definition a doubly filtered, graded  chain complex over $\ff_2$.  Thus a set of filtered generators can be illustrated on a grid with the coordinates representing the filtration levels and the grading marked.  There is an action of $\zz$ on the complex, and if we let $U$ be the generator, this  makes the complex a $\ff_2[U, U^{-1}]$--module.  The action of $U$ on the complex  lowers filtration levels by 1 and gradings by 2.

We now show that  $CFK^\infty(S^3, T_{3,5})$ is  as illustrated  in Figure~\ref{hf35-4.pdf}.   In order to find this decomposition, we start by focusing on the  central column (for which the top-most generator is at filtration level $j = 4$ and is labeled with its  grading 0).  
The vertical column, $i = 0$, represents the sub-quotient complex  $\widehat{CFK}(S^3 ,T_{3,5})$.  We begin by explaining why it appears as it does in the illustration.   According to~\cite[Theorem 1.2]{os3}, since for torus knots there is an integer surgery that yields a lens space, $\widehat{HFK}(S^3, T_{3,5}, j)$, the quotients of the $j$-filtration level by the $(j-1)$--filtration level is completely determined by the Alexander polynomial, $$\Delta_{T_{3,5}}(t) = 1 -(t^{-1}+t) +(t^{-3} + t^3) - (t^{-4} + t^4).$$  This explains the location of the  generators of  $\widehat{CFK}(S^3 ,T_{3,5})$.  Similarly,~\cite{os3} determines the grading of the generators.  The fact   the complex $\widehat{CFK}(S^3 ,T_{3,5})$ is a filtration of the complex $\widehat{CF}(S^3 )$ which has homology $\ff_2$ with its generator at  grading level 0, forces the vertical arrows, presenting the boundary maps, to be as illustrated. 
To build the $ {CFK}^\infty$ diagram from the $\widehat{CFK}$ diagram, we first apply the action of $U$ to fill in the generators as well  as the all the vertical arrows.  We next note that the homology groups $\widehat{HFK}(T_{3,5},i)$ can be computed using the horizontal slice $j=0$ instead of the vertical slice, and this forces the existence of the horizontal arrows as drawn.  With this much of the diagram drawn, and the action of $U$ lowering grading by $2$, the gradings of all the elements in the diagram are determined.  Finally, we note that the fact that the boundary map lowers gradings by 1 rules out the possibility of any other arrows.

\begin{figure}[h]
\fig{.4}{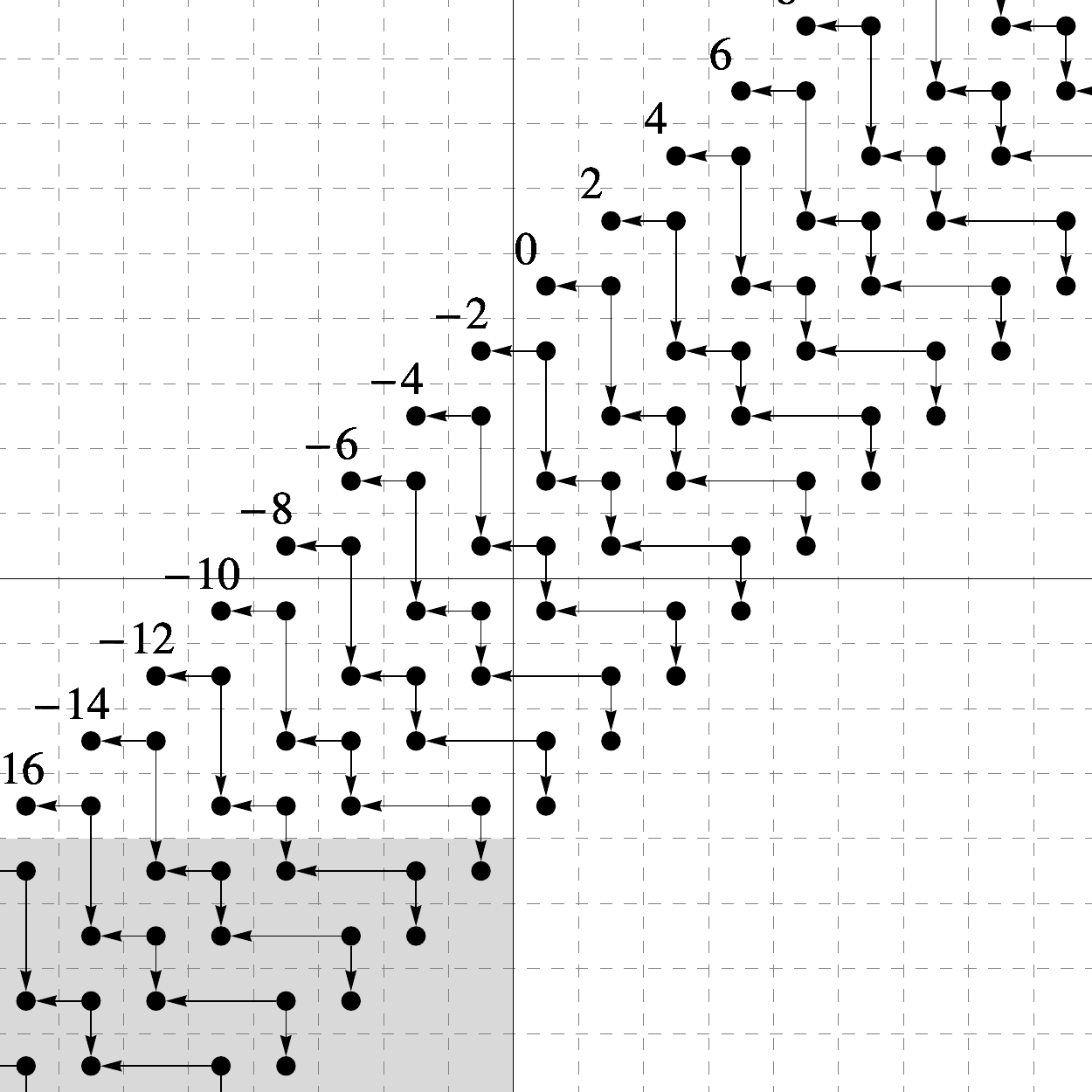}
\caption{}\label{hf35-4.pdf}
\end{figure}

According to~\cite{os4}, the complex  $CFK^+(S^3_{15}(T_{3,5}),s)$, for $-7 \le s \le 7$ is given by the quotient 
$$CFK^\infty(S^3 , T_{3,5} ) / CFK^\infty(S^3, T_{3,5} )_{i <0, j<s}[- \eta],$$ where the quotienting subgroup is shaded in the diagram  for $s=-4$.  Here $\eta$ is a grading shift:
$$\eta = \frac{-(2s-15)^2 +15}{60}.$$ 
By definition, the  $d$--invariant is  the minimal grading  among all classes in  the group $HFK^+(S^3_{15}(T_{3,5}),s )$ which are in the image of $U^n$ for all $n$.  From the diagram, without shifting the gradings, we see this minimum for $HFK^+(S^3_{15}(T_{3,5}),-4 )$ is $-8$: one generator of grading level $-10$ has been killed, and all such generators are homologous.  The values for all \Spc--structures, $s = -7, -6, \ldots , 6,7$ are given in order as $$\{-14, -12,-10,-8,-8,-6,-4,-4,-2,-2,-2,0,0,0,0\}.$$
After the grading shift, the values are all of the form $a_i/30$, where, in order, the $a_i$ are:
$$\{-7,-3,5,17,-27,-7,17,-15,17,-7,-27,17,5,-3,-7\}.$$\vskip.05in
Finally, to compute $\bar{d}$, we subtract  $-15/30$ (the value for the Spin structure) to each entry, and find that the values of $\bar{d}$ are given by $b_i/30$ for the following values of $b_i$ in order.
$$\{8,12,20,32,-12,8,32,0,32,8,-12,32,20,12,8\}.$$
We have listed these values in the chart of Figure~\ref{dt35}, in which we write each value of $s$ as $5a+3b \mod 15$ for $-1\le a \le 1$ and $-2\le b \le 2$.  
\vskip.1in
\begin{figure}[h]
$$\begin{array}{c|c|c|c|c|c|} 
& b= -2 &b= -1 & b=0 & b= 1&b= 2 \\
\hline
 a=  1  & 32 &8 & \bf  20 &  8& 32 \\
\hline 
 a=0   &\bf   12&\bf   -12&\bf   0& \bf  -12 &\bf   12\\
\hline 
 a=  -1  & 32 &  8&  \bf 20& 8& 32\\
\hline
\end{array}$$
\caption{$30\,\bar{d}(S^3_{15}(T_{3,5}), 5a+3b)$}\label{dt35}
\end{figure}
\vskip.1in

Since $S^3_{15}(T_{3,5})$ is the connected sum of lens spaces, Theorem~\ref{obstructthm} predicts a pattern in the chart: each element should be the sum of the entries of its projection on the the main axes.  This is the case.  Notice for instance that the top right entry 32 in position $(a,b) = (1,2) \in \zz_3 \oplus \zz_5$ (which represents $1(5) +2(3) = 11 \in \zz_{15}$), is the sum of the entries in positions $(2,0)$ and $(0,1)$, 12 and 20, respectively. 
 

\subsection{$\bf \bar{d}(S^3_{15}(T_{3,5} \cs D), \spinc)$.}  In order to compute the $\bar{d}$--invariants that are associated to surgery on the connect sum, we first must compute  $CFK^\infty$ for the connected sum of knots.  The complex $CFK^\infty(T_{2,3})$ is illustrated in Figure~\ref{hf23figure}, and it follows from~\cite{hedden} that, modulo acyclic subcomplexes, the homology of the double $D(T_{2,3})$ is the same.

\begin{figure}[h]
\fig{.3}{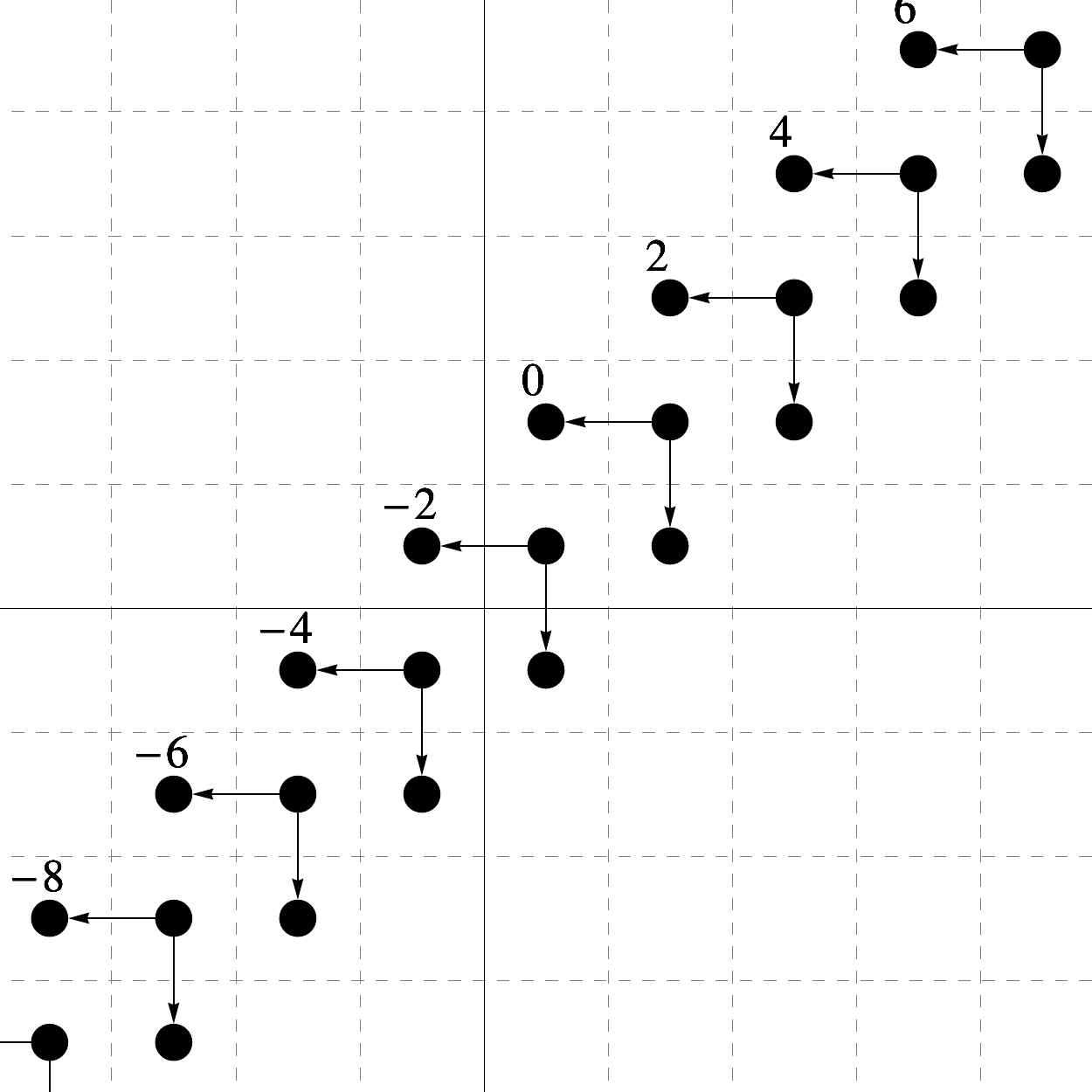}
\caption{}\label{hf23figure}
\end{figure}

At this point we need to analyze the tensor product,  $$C = CFK^\infty(T_{3,5}) \otimes_{\ff[U,U^{-1}]} CFK^\infty(T_{2,3}).$$ This complex is fairly complicated, containing 21 generators, but it is easily seen that it contains a subcomplex $C'$ as illustrated in Figure~\ref{hf35doublefigure}.  This subcomplex carries the homology of the overall complex, but does not contain all generators of a given grading.  However, it has the following property.  

\begin{theorem} The complex $C_{i<m, j<n}$ contains a generator of grading 0 if and only if $C'_{i<m, j<n}$ contains a generator of grading 0.  In particular, $d$--invariants for $C$ can be computed using $C'$.
\end{theorem}

\begin{figure}[h]
\fig{.5}{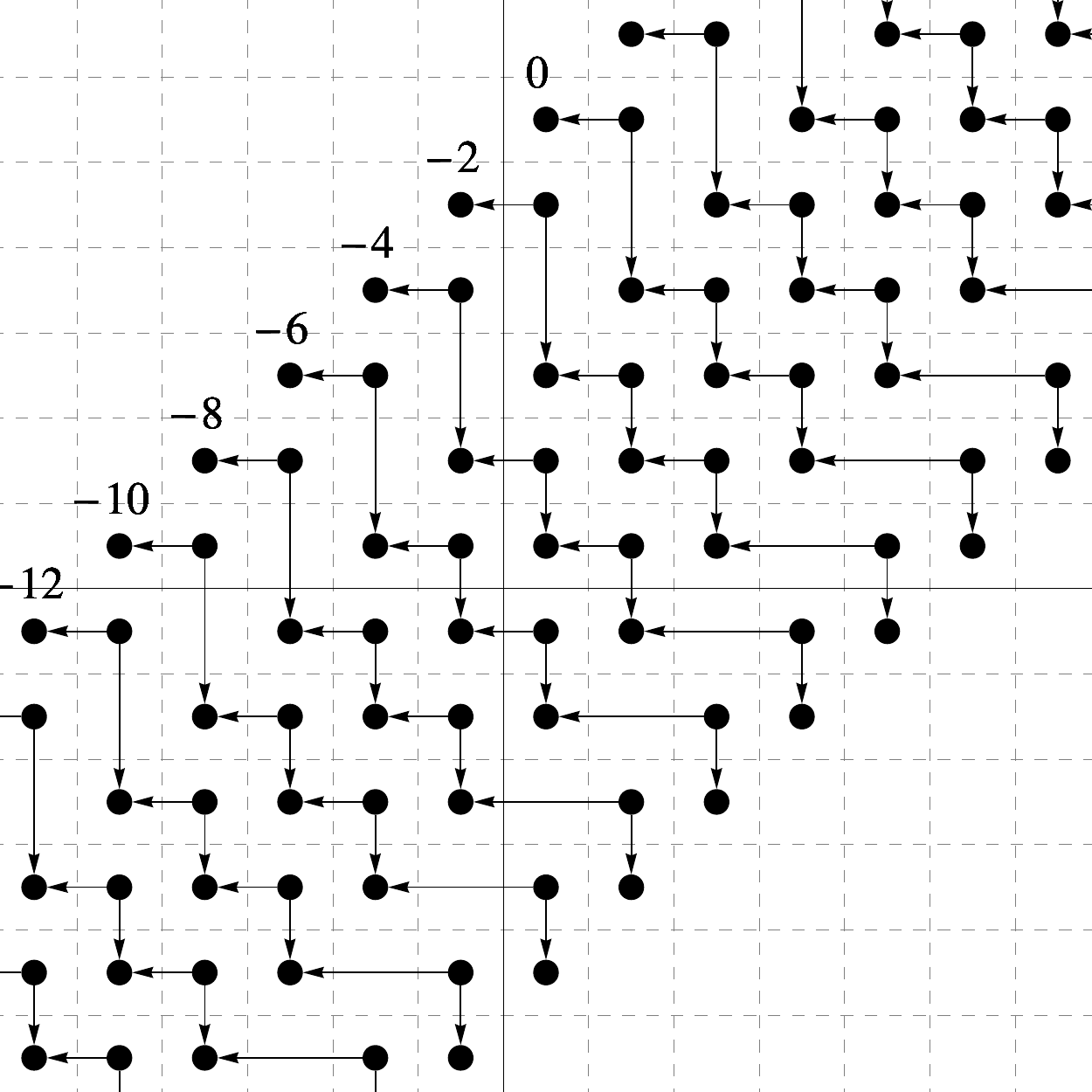}
\caption{}\label{hf35doublefigure}
\end{figure}

Using this diagram to compute the minimal gradings of classes in $$CFK^\infty(T_{3,5} \cs D ) / CFK^\infty(T_{3,5}\cs D)_{i <0, j<s} $$ for $-7 \le s \le 7$ we get the following:
$$\{ -14, -12, -10,-10,-8, -6, -6, -4, -4, -2, -2, -2, 0,0,0\}.$$  After shifting gradings by $-\eta$, the values are of the form $a_i/30$, where the $a_i$ are, in order, 
$$\{ -7, -3, 5, -43, -27, -7, -43, -15, -43, -7, -27, -43, 5, -3, -7\}.$$  To compute $\bar{d}$, we add $15/30$ to each term, yielding the values $b_i/30$, where the $b_i$ are:
$$\{ 8,12,20,-28,-12,8,-28,0,-28,8,-12,-28,20,12,8\}.$$ 
We can arrange these in a chart shown in Figure~\ref{dt35D}.
\begin{figure}[h]
$$\begin{array}{c|c|c|c|c|c|} 
& b= -2 &b= -1 & b=0 & b= 1&b= 2 \\
\hline
 a=  1  &  \underline{-28} &8 & \bf  20 &  8& \underline{-28}  \\
\hline 
 a=0   &\bf   12&\bf   -12&\bf   0& \bf  -12 &\bf   12\\
\hline 
 a=  -1  &   \underline{-28}  &  8&  \bf 20& 8&   \underline{-28} \\
\hline
\end{array}$$
\caption{$30\,\bar{d}(S^3_{15}(T_{3,5} \# D), 5a +3b)$}\label{dt35D}
\end{figure}

Notice that the entries on the axes are  unchanged, but the underlined entries are no longer the sum of the values of the projections; that is, $-28 \ne 12 + 20$.  Thus, according to Theorem~\ref{obstructthm}, this manifold is not  $\qq$--homology cobordant to any manifold of the form $M_3 \cs M_5 \cs M_q$. 


\subsection{Second Example}

As a second example  we consider the case of $S^3_{35}(T_{5,7})$ and $S^3_{35}(T_{5,7}\cs D)$  and illustrate the analogous charts as above (this time multiplied by 70 to clear denominators). The first chart, Figure~\ref{dt57} necessarily demonstrates additivity, the second, in Figure~\ref{dt57D},  upon examination does not.  This becomes more apparent by considering the third chart, in Figure~\ref{dt57Ddiff}, formed as the difference of the first two, but not multiplied by 70.  The underlined entries illustrate the failure of additivity.  Considering this difference is a simplifying approach of the general proof in the next section. 
\begin{figure}
$$\begin{array}{c|c|c|c|c|c|c|c|} 
& b= -3 &b= -2 & b=-1 & b= 0&b= 1& b= 2&b= 3 \\
\hline
 a=  2  & -68 & -108 & -48 &\bf -28 & -48 & -108 & -68 \\
\hline 
 a=1   &-12 & -52 & 8 & \bf 28 & 8 & -52 & -12 \\
\hline 
 a=  0  &\bf -40 &\bf -80 &\bf -20 &\bf 0 &\bf -20 &\bf -80 &\bf -40 \\
\hline
 a=  -1  &-12 & -52 & 8 &\bf 28 & 8 & -52 & -12 \\
\hline
 a=  -2  & -68 & -108 & -48 &\bf -28 & -48 & -108 & -68\\
\hline
\end{array}$$
\caption{$70\,\bar{d}(S^3_{35}(T_{5,7}), 7a +5b)  $}\label{dt57}
\end{figure}


\begin{figure}
$$\begin{array}{c|c|c|c|c|c|c|c|} 
& b= -3 &b= -2 & b=-1 & b= 0&b= 1& b= 2&b= 3 \\
\hline
 a=  2  &\underline{72} & \underline{32} & 92 &\bf 112 & 92 & \underline{32} & \underline{72} \\
\hline 
 a=1   &128 & 88 & 8 &\bf 28 & 8 & 88 & 128 \\
\hline 
 a=  0  &\bf100 &\bf 60 &\bf -20 &\bf 0 &\bf -20 &\bf 60 &\bf 100\\
\hline
 a=  -1  &128 & 88 & 8 &\bf 28 & 8 & 88 & 128 \\
\hline
 a=  -2 &\underline{72} & \underline{32} & 92 &\bf 112 & 92 & \underline{32} & \underline{72}\\
\hline
\end{array}$$
\caption{$ 70\,\bar{d}(S^3_{35}(T_{5,7}\#D), 7a +5b)$}\label{dt57D}
\end{figure}


\begin{figure}
$$\begin{array}{c|c|c|c|c|c|c|c|} 
& b= -3 &b= -2 & b=-1 & b= 0&b= 1& b= 2&b= 3 \\
\hline
 a=  2  &\underline{2} & \underline{2} &2 &  2 & 2 & \underline{2} & \underline{2} \\
\hline 
 a=1   &2 & 2 &0 &  0 & 0 &2& 2 \\
\hline 
 a=  0  &  2 &  2 &  0 &  0 &  0 &  2 &  2\\
\hline
 a=  -1  &2 & 2&0 &  0 &0 & 2 & 2 \\
\hline
 a=  -2 &\underline{2} &\underline{2}&2&  2 & 2 & \underline{2} & \underline{2}\\
\hline
\end{array}$$
\caption{$\bar{d}(S^3_{35}(T_{5,7}\#D), 7a +5b) - \bar{d}(S^3_{35}(T_{5,7}), 7a +5b)$}\label{dt57Ddiff}
\end{figure}


\section{Topologically split examples, general case.}\label{secexamplestopsplit1}  
We now wish to generalize the examples of the previous section.  To do so, we begin by choosing an infinite set of integers $\{p_i\}$ with the following properties: (1) all $p_i$ are odd; (2)
the full set of integers   $\{p_i, p_i +2\}$ is pairwise relatively prime; and, (3) each $p_i$ and $p_i +2$ is square free.  The existence of such a set is demonstrated in Appendix~\ref{appendpi}, and throughout this section we assume all $p$ are selected from this set.
In the previous example we needed to track grading shifts.  It will simplify our discussion if we avoid dealing the grading shifts  as follows:  define $\tilde{d}(S^3_n(K), s) = d(S^3_n(K),s) +\eta$.  That is, $\tilde{d}$ is computed as is the $d$--invariant, except without the grading shift, the induced grading on \[
CFK^+(S^3_N(K),s)=CFK^\infty(S^3,K)/CFK^\infty(S^3,K)_{\{i<0,j<s\}}
\]
Since $p$ is odd, we can write $ p= 2n+1$ and let $q = p+2 = 2n +3$.  Our manifolds of interest are $S^3_{pq}(T_{p,q})$ and   $S^3_{pq}(T_{p,q} \cs D)$.  We collect here the results of a few elementary calculations.
\begin{theorem}\label{pqcalcs} $ \ $
\begin{enumerate}

\item The surgery coefficient is $$pq = 4n^2 + 8n +3.$$ \vskip.05in

\item The three-genus satisfies $$g(T_{p,q}) = 2n(n+1) = 2n^2+2n  \hskip.2in \text{and} \hskip.2in g(T_{p,q} \cs D)     = 2n^2+2n +1.$$\vskip.05in

\item \Spc--structures are parameterized by $s$, with $$-(2n^2 +4n+1)  \le s \le (2n^2 +4n+1).$$\vskip.05in

\item  Generators of $\widehat{CFK}(T_{p,q})$ have filtration level $j$, where $$-2n(n+1) \le j \le 2n(n+1).$$\vskip.05in
\end{enumerate}
\end{theorem}
The main result of this section is the following.

\begin{theorem}\label{pqtorus}
$\bar{d}(S^3_{pq}(T_{p,q} \cs D), s)$ does not satisfy additivity as given in Theorem~\ref{obstructthm}.
\end{theorem}
\begin{proof}
The space $S^3_{pq}(T_{p,q})$ satisfies the additive property as in Theorem~\ref{obstructthm}.  Suppose that  $S^3_{pq}(T_{p,q} \cs D)$ also satisfies  additivity property. Then the difference $\bar{d}(S^3_{pq}(T_{p,q}),(a,b))-\bar{d}(S^3_{pq}(T_{p,q} \cs D),(a,b))$ also satisfies the additivity property. We denote this difference by  $\bar{d}'(a,b)$ or $\bar{d}'(aq+bp)$. 
Note that it is unnecessary to add the grading shift $\eta$ to the amount we get from the diagram  when computing either of the values $\bar{d}(S^3_{pq}(T_{p,q}),(a,b))$ or $\bar{d}(S^3_{pq}(T_{p,q} \cs D),(a,b))$ since they have the same grading shift. Namely,
\begin{align*}
\bar{d}'(a,b) &= \tilde{d}(S^3_{pq}(T(p,q)),(a,b)) -\tilde{d}(S^3_{pq}(T_{p,q} \cs D),(a,b))\\
&-\tilde{d}(S^3_{pq}(T_{p,q}),0) +\tilde{d}(S^3_{pq}(T_{p,q} \cs D),0).
\end{align*}
From our choice of $p$ and $q$, we have $(n+1)p+(-n)q=1$. Thus, the additivity property implies the equality
\[
\bar{d}'(1)=\bar{d}'((n+1)p)+\bar{d}'(-nq),
\]
or equivalently,
\begin{align}
\tilde{d}(S^3_{pq}(T_{p,q}),1) &-\tilde{d}(S^3_{pq}(T_{p,q} \cs D),1) \nonumber \\
& =\tilde{d}(S^3_{pq}(T_{p,q}),(n+1)p) -\tilde{d}(S^3_{pq}(T_{p,q} \cs D),(n+1)p) \label{eqn:tilded}\\
&+\tilde{d}(S^3_{pq}(T_{p,q}),-nq) -\tilde{d}(S^3_{pq}(T_{p,q} \cs D),-nq) \nonumber \\
&-\tilde{d}(S^3_{pq}(T_{p,q}),0) +\tilde{d}(S^3_{pq}(T_{p,q} \cs D),0). \nonumber
\end{align}
Since $(n+1)p=2n^2+3n+1$ lies between the genus of  $T(p,q)$ (and of $T_{p,q} \cs D$) and the upper bound on the parameters for the \Spc--structures:
$$  2n^2 +2n +1 < 2n^2+3n +1 < 2n^2 +4n +1,$$
the values of the $\tilde{d}$--invariants are easily seen to be 0.
On the other hand, the number  $-nq$  is greater than the lower bound on the parameters for the \Spc--structures and less than the negative of the genus:
$$            -(2n^2 + 4n +1) < -(2n^2 +  3n) < -(2n^2  + 2n  + 1  )$$  and thus one sees that the $\tilde{d}$--invariants  take the same value $-2s = 2(2n^2+3n)$ for both $ T_{p,q}$ and $T_{p,q} \cs D$. 

Thus, in contradicting additivity, it remains to show that the equality $$ 
\tilde{d}(S^3_{pq}(T_{p,q}),1)  -\tilde{d}(S^3_{pq}(T_{p,q} \cs D),1) = 
-\tilde{d}(S^3_{pq}(T_{p,q}),0) +\tilde{d}(S^3_{pq}(T_{p,q} \cs D),0)$$ 
does not hold.

Now we will compute $\tilde{d}$ of both spaces for \Spc--structures $0$ and $1$.
Observe that within width 1 from the diagonal $j=i$, the complex $CFK^\infty(S^3,T_{p,q})$ looks like $CFK^\infty(S^3,T_{2,3})$ if $n$ is odd, or $CFK^\infty(S^3,T_{2,5})$ if $n$ is even.  This depends on the fact that near the origin the complex $CFK^\infty(S^3,K) $ looks like that of the $(2,k)$--torus knots.  In Appendix~\ref{torusknotpoly} we prove that the Alexander polynomial of $T_{p,p+2}$ is of the form   1+ $\sum_{i>0} a_i(t^{-i} + t^{i})$ where $a_i = \pm 1$ for $i\le (p-1)/2$.  As in the example of the previous section, this determines the ``zig-zag'' feature of the $CFK^\infty$ complex near the origin.  Tensoring with the trefoil complex does not alter this pattern.

The generators of the same grading $2l$ of $[x,-1,0]$ if $n$ is odd (or, $[x,0,0]$ if $n$ is even) lies above the anti-diagonal $i+j=-1$ (or, $i+j=0$). So, in order to compute $\tilde{d}(S^3_{pq}(T(p,q)),s)$ for $s=0,1$, we may assume in the computations that the complex we are considering is one of 
$$
\begin{cases}
CFK^\infty(S^3,T_{2,3}) & \text{if $n$ is odd}, \\
CFK^\infty(S^3,T_{2,5}) & \text{if $n$ is even}.
\end{cases}
$$
It is now easy to compute
$$
\tilde{d}(S^3_{pq}(T_{p,q}),s)=
\begin{array}{|c|c|c|} \hline
s & n \text{ odd} & n \text{ even} \\ \hline
1 & 2l+2 & 2l \\ \hline
0 & 2l & 2l \\ \hline
\end{array}.
$$
Near the diagonal $j=i$, the complex $CFK^\infty(S^3,T_{p,q} \cs D)$ looks like:
$$
\begin{cases}
CFK^\infty(S^3,T_{2,5}) & \text{if $n$ is odd}, \\
CFK^\infty(S^3,T_{2,3})[-2] & \text{if $n$ is even}.
\end{cases}
$$
The grading of $[x,-1,0]$ is $2l-2$ if $n$ is even and the grading of $[x,0,0]$ is $2l$ if $n$ is odd. Thus, we have
$$
\tilde{d}(S^3_{pq}(T_{p,q} \cs D),s)=
\begin{array}{|c|c|c|} \hline
s & n \text{ odd} & n \text{ even} \\ \hline
1 & 2l & 2l \\ \hline
0 & 2l & 2l-2 \\ \hline
\end{array}.
$$
We see that
$$
\tilde{d}(S^3_{pq}(T_{p,q}),s)-\tilde{d}(S^3_{pq}(T_{p,q} \cs D),s)=
\begin{array}{|c|c|c|} \hline
s & n \text{ odd} & n \text{ even} \\ \hline
1 & 2 & 0 \\ \hline
0 & 0 & 2 \\ \hline
\end{array}.
$$
This shows that (\ref{eqn:tilded}) cannot be satisfied. We conclude that the space $S^3_{pq}(T_{p,q} \cs D)$ does not satisfy the additive property of Theorem~\ref{obstructthm}.
\end{proof}
\subsection{The image of $\calk$ in $\Theta^3_\qq / \Phi(\oplus_{p \in \calp} \Theta^3_{\zz[1/p]})$ is infinite.}  This follows from the following result.

\begin{theorem}The spaces $N_{p,q} = S^3_{pq}(T_{p,q} \cs D) \cs -   S^3_{pq}(T_{p,q} ) \in \calk$ are distinct in the quotient $\Theta^3_\qq / \Phi(\oplus_{p \in \calp} \Theta^3_{\zz[1/p]})$.  
\end{theorem} 

\begin{proof} Observe that $S^3_{pq}(T_{p,q} \cs D) \cs -   S^3_{pq}(T_{p,q} ) \in \calk$, since the knots are topologically concordant.     We next  observe that these manifolds have the property that no linear combination with all coefficients $\pm 1$ is trivial in the quotient.  Suppose that some such linear combination was trivial.  Then focusing on any particular pair $(p,q)$, we would have that $S^3_{pq}(T_{p,q} \cs D) \# M_p \#M_q \# M_m = \partial X$ for a rational homology ball $X$, where the order of $M_p$ is  a product of prime factors of $p$, the order of $M_q$ is  a product of prime factors of $q$, and the order of $M_m$ is relatively prime to $pq$.  (This uses the fact that $S^3_{pq}(T_{p,q} ) $ does split as a connected sum.)

The existence of this connect sum decomposition implies the  additivity for $d$--invariants of  $S^3_{pq}(T_{p,q} \cs D)$ in a way that contradicts  Theorem~\ref{pqtorus}. 
\end{proof}

\section{Knot concordance}\label{concordance}
We denote by $\calc$ the classical smooth knot concordance group.  Levine~\cite{levine} defined the algebraic concordance group  $\calg$ and the rational algebraic concordance group, $\calg^\qq$.  He also defined a surjective homomorphism $\calc \to \calg$, proved that natural map $\calg \to \calg^{\qq}$ is injective, and proved that  $\calg^{\qq}$ is isomorphic to an infinite direct sum of groups isomorphic to $\zz, \zz_2$ and $\zz_4$.  He also proved that the image of $\calg$ in $\calg^\qq$ is isomorphic to a similar infinite direct sum.  
In~\cite{levine} it is observed that $\calg^\qq$ has a natural decomposition as a direct sum $\oplus \calg^\qq_{p(t)}$, where the $p(t)$ are symmetric irreducible rational polynomials.  We will not present the details here, but note that if the Alexander polynomial of $K$, $\Delta_K(t)$, is irreducible, then the image of $K$ in $\calg^\qq$ is in the $\calg^\qq_{\Delta(t)}$ summand.
Stoltzfus~\cite{stoltzfus} observed that the algebraic concordance group $\calg$ does not have a similar splitting.  Thus, there is not an immediate analog in concordance for the decompositions we have been studying for homology cobordism.  However, he did prove that in some cases such a splitting exists.  The following, Corollary 6.5 from~\cite{stoltzfus}, is stated in terms of knot concordance, but given the isomorphism of higher dimensional concordance and $\calg^\zz$, the same splitting theorem holds in the algebraic concordance group.

\begin{theorem} If $\Delta_K(t)$ factors as  $p(t)q(t)$ with $p(t)$ and $q(t)$ symmetric  and the resultant \emph{Res}$(p(t), q(t)) = 1$, then $K$ is concordant to a connected sum $K_1 \cs K_2$, with $\Delta_{K_1}(t) = p(t)$  and $\Delta_{K_2}(t) = q(t)$.
\end{theorem}

Here we observe that this result does not hold in dimension 3.\vskip.05in
\noindent{\bf Example.} Consider the ten  crossing knot $K = 10_{5}$.  It has Alexander polynomial $$\Delta = (1 - t + t^2) (1 - 2 t + 2 t^2 - t^3 + 2 t^4 - 2 t^5 + t^6).$$ These two factors are irreducible and have resultant 1. 

\begin{theorem} The knot $10_5$ is not concordant to any connected sum $K_1 \cs K_2$ where $\Delta_{K_1} = 1 - t + t^2  $ and $\Delta_{K_2}  =   1 - 2 t + 2 t^2 - t^3 + 2 t^4 - 2 t^5 + t^6$.
\end{theorem}
\begin{proof} The 2-fold branched cover of $K$ is the lens space $L(33,13)$.  If the desired concordance existed, then $L(33,13)$ would split in rational cobordism as a connected sum $M_3 \cs M_{11}$, with $H_1(M_3) = \zz_{3}$ and $H_1(M_{11}) = \zz_{11}$.  In order to compute the relevant $d$--invariants, one first identifies $\spinc_6$ as the Spin--structure $\spinc_*$ by computing that  the value of $d(L(33,13), \spinc_6)= 33$, a value that is not attained by any other \Spc--structure.  The values of the $d$--invariants, $d(L(33,13), (a,b)\cdot \spinc_*) - d(L(33,13),   \spinc_*)$ for $(a,b) \in \zz_3 \oplus \zz_{11}$ are given in the chart in Figure~\ref{dl33} (multiplied by 33 to clear denominators).

\begin{figure}[h]
$$\begin{array}{c|c|c|c|c|c|c|c|} 
& b= 0 &b=  1 & b=2 & b= 3&b= 4& b= 5  \\
\hline
 a=2   &\bf 22 &10 &40 & -20  & 28 & -14   \\
\hline 
 a=  0  &\bf 0 &\bf 54 &\bf 18 &\bf 24  &\bf 6 &\bf 30  \\
\hline
 a=  -1  & \bf 22 & 10 & 40 & 46 & 28 & 52   \\
\hline
\end{array}$$
\caption{$33\,d(L(33,13), 11a+3b) $}\label{dl33}
\end{figure}

The next chart, in Figure~\ref{dl33diff}, presents the values 
\begin{align*}
\delta(L(33,13),(a,b)) &= d(L(33,13),(a,b)) - d(L(33,13),(a,0)) \\
 &- d(L(33,13),(0,b)) + d(L(33,13),(0,0)).
\end{align*}

\begin{figure}[h]
$$\begin{array}{c|c|c|c|c|c|c|c|} 
& b= 0 &b=  1 & b=2 & b= 3&b= 4& b= 5  \\
\hline
 a=2   & \bf 0 & 2 & 0 &  2  & 0 & 2   \\
\hline 
 a=  0  &\bf 0 &\bf 0 &\bf 0 &\bf 0  &\bf 0 &\bf 0  \\
\hline
 a=  -1  &\bf  0 & 2 &  0 &0 & 0 & 0   \\
\hline
\end{array}$$
\caption{$ \delta(L(33,13), (a,b)) $}\label{dl33diff}
\end{figure}
The presence of the nonzero entries implies the nonsplittability of the manifold, as desired. 
\end{proof}

\noindent{\bf Note.} In unpublished work~\cite{livingst} the second author constructed similar  but much more complicated  examples in the topological category.

\section{Topologically trivial bordism}\label{toptrivialbordism}

In~\cite{hlr} the quotient $\Theta^T_{\qq, spin} / \Theta^I_{\qq, spin}$ was studied.  Here, the cobordism group has been restricted to  spin 3--manifolds and spin bordisms  which have the rational homology of $S^3$.  The notation $\Theta^T_{\qq, spin}$ denotes the subgroup generated by representatives  which bound topological homology balls and $ \Theta^I_{\qq, spin}$ is generated by those that are cobordant to $\zz$--homology spheres.  (Note we have changed the notation from that of~\cite{hlr} to be consistent with the results of the current paper.  There is a similar result in~\cite{hlr} replacing $(\qq, spin)$, with $\zz_2$.  (Recall that every $\zz_2$ homology sphere is spin.)

Here we observe that Theorem~\ref{lemmaextendspin} permits us to generalize this result, eliminating the need to constrain the cobordism group to being spin or to use $\zz_2$ coefficients.  Let $\Theta^T_\qq$ denote the subgroup of $\Theta^3_\qq$ generated by rational homology spheres that are trivial in the topological rational cobordism group, that is, the kernel of $\calk$.

\begin{theorem}  The quotient group $\Theta^T_\qq / \Theta^3_\zz$ is infinitely generated. 
\end{theorem}

We outline how the argument in~\cite{hlr} can be generalized.

In~\cite{hlr} there is a family of rational homology spheres constructed, $M_{p^2}$, for an infinite set of primes $p$.  These are constructed so that they  bound topological balls.  The proof of the theorem consists of showing  that no linear combination $N = \cs_i a_iM_{p_i^2} \cs M_0$ bounds a spin rational homology ball (or $\zz_2$ homology ball) $W$, where $M_0$ is  a $\zz$--homology sphere.  The existence of a unique Spin--structure was used to identify \Spc\  of the relevant manifolds with the second homology.

If all $p$ are odd, then there is a unique \Spc--structure on $N$ and according to Theorem~\ref{lemmaextendspin}, it is the restriction of a \Spc--structure on $W$.  
Given this, Proposition 2.1 of~\cite{hlr}, which required that $W$ be spin, continues to apply to identify the \Spc--structures on 
$N$ which extend to $W$ with a metabolizer of the linking form on   $H_1(N)$.  That identification is what is used to obstruct the existence of $W$ via $d$--invariants, as described in Thoerem 3.2 of~\cite{hlr}.  Thus, the remainder of the proof goes through as in that paper.

\appendix
\section{ Finding the $p_i$}\label{appendpi}
The proof of Theorem~\ref{pqtorus} requires a sequence of odd pairs $\{ p_i, p_i +2\}$ so that  the elements of the full set of $\{p_i\} \cup \{p_i +2\}$ are pairwise relatively prime and square free.  Since $p_i $ and $p_i +2$ are relatively prime, we need to choose the $p_i$ so that the set of all elements of $\{ p_i(p_i +2)\}$ are pairwise relatively prime and each element is square free.  If we let $p_i = n_i -1$, then $p_i(p_i+2) = n_i^2 -1$, and so we are seeking an infinite sequence of positive integers $\{n_i\}$ such that: 
\begin{enumerate}
\item  $n_i$ is even for all $i$.\vskip.05in
\item All elements of $\{n_i^2 -1\}$ are relatively prime.\vskip.05in

\item Each $n_i^2 -1 $ is square free.
\end{enumerate}
In Section~\ref{inftwotor} we need a sequence of integers $n_i$ such that $n_i = 0 \mod5$  with the property that the integers $20n_i^2 +8n_i +1$ are relatively prime and square free.
Here is a theorem that covers both cases.

\begin{theorem} Let $f(x)\in \zz[t]$ be an   quadratic polynomial with constant term 1 that is not the square of a linear polynomial.  Let $\alpha$  be a fixed integer and $s_n= \alpha n$ be an arithmetic sequence.   There exists an infinite set of $s_i$ such that values of $f(s_i)$ are pairwise relatively prime and square free. 
\end{theorem}

\begin{proof}  It is known that if  $g(n)$ is a quadratic polynomial that is not a square of a linear polynomial and which has the  property that its coefficients have greatest common divisor one, then $g(n)$ is square free for an infinite set of $n$ (see, for example,~\cite{erdos}).  
We wish to construct the sequence of $s_i$ inductively.  To find $s_1$, let $f_1(n) = f(\alpha n)$, which is irreducible with constant term one.  Choose $n_1$ so that $f_1(n_1)$ is square free.  Let $s_1 = \alpha n_1$.
Assume that $s_i$ has been defined for $i < k$.  We  find $s_{k} $ with the desired properties as follows.   Let $P = \prod_{i=1}^{k-1} f(s_i)$.  Consider the function $f_k(n) = f(\alpha P n)$.  Again, this polynomial is irreducible with constant term one, so there exists an $n_k$ for which $f_k(n_k)$ is square free.  Since $f_k(n_k) = f(\alpha P n_k)$, we let $s_k = \alpha P n_k$.  Notice   that for each prime divisor $p$ of $P$, $f(\alpha P n) = 1 \mod p$, since evaluating $f$ at  $\alpha P n$ gives a quadratic polynomial  in $n$, with the quadratic term and linear term divisible by $P$ and the constant term one.  It follows that $f(s_k)$ is relatively prime to all $f(s_i), i < k$.
\end{proof}


\section{The Alexander polynomial of $T_{p,p+2}$.}\label{torusknotpoly}
Normalized to be symmetric, the Alexander polynomial of  a knot can be written in the form $\Delta_K(t) = a_0 +\sum_{i=1}^{n} a_i(t^{-i} + t^i)$, where $a_0 + 2\sum a_i = \pm 1$.  In Section~\ref{secexamplestopsplit1} we use the following fact.

\begin{theorem} If $K = T_{p,p+2}$ with $p$ odd then  $$\Delta_{T_{p,p+2}}(t) = a_0 +\sum_{i=1}^{(p^2-1)/2} a_i(t^{-i} + t^i),$$ where $a_i = \pm 1$  for $i \le (p-1)/2.$
\end{theorem}

\noindent{\bf Note.} With more care, all the coefficients or $ \Delta_{T_{p,p+2}}(t)$ can be described in closed form.

\begin{proof}
As a polynomial (as opposed to the normalized Laurent polynomial) with nonzero constant term, the Alexander polynomial of $T_{p,q}$ is $(1-t^{pq})(1-t)/(1-t^p)(1-t^q)$.  Expanding each term of the denominator in a power series and noting that multiplying by the $t^{pq}$ term in the numerators does not affect terms of the product of degree less than $2g = (p-1)(q-1)$, the degree of the Alexander polynomial, we can focus on the expression:
$$(1-t) ( 1 + t^p + t^{2p} + t^{3p}\cdots)(1+t^q + t^{2q} +\cdots),$$ which we write as the product $$(1-t)\sum_{i=0}^{\infty} b_i t^i.$$  Here $b_i $ is the number of solutions to $xp + y q = i$, with $x, y \ge 0$.
In the case of interest, $q = p+2$ and the genus $g = (p^2-1)/2$.  We will now show  that for $i$ in the range $g- A \le i \le g$, the values $b_i$ are alternately 0 and 1, where $A$ is a constant to be determined.  Thus, using the fact that the Alexander polynomial is symmetric, upon multiplying by $(1-t)$ we have the coefficients of the Alexander polynomial are all $\pm 1$ near $t^g$.  
To show that the coefficients $b_i$ alternate between 0 and 1 for $g - A \le i \le g$, we first observe that in a given range of $i$, all $b_i \ge  1$ for $i$ even.   To see this, write $p= 2n+1$ and $q= 2n+3$; thus $g = 2n^2 + 2n$.  Consider the sum $$ \frac{n+j}{2}p +\frac{n-j}{2}q = 2n^2 +2n -j,$$ where $j$ is selected to have the same parity as $n$. (We require here that $j \le n$, that is, we need $A \le \frac{p-1}{2}$.)
To complete the argument, we next observe that the difference $|b_{i} - b_{j}| \le 1$ if $|i-j| \le 1$.  Suppose otherwise.   That is, suppose that there are {\it distinct} nonnegative solutions to equations:
$$ x p + y q = i$$
and $$x' p +y'q = j$$
with $x, y, x' , y' \ge 0$, $|i-j|\le 1$, and $i, j \le g$.
The conditions that $i\le g$ and $y \ge 0$ imply that $xp \le g = (pq - p -q -1)/2$, which imply that $x < (q-1)/2$.
We first consider the case that $i \ne j$.  After  possibly rordering, the difference would give 
$$(x- x') p + (y-y')q =1.$$
One solution to this equation is $$\frac{q-1}{2} p -\frac{p-1}{2}q = 1.$$ Every other solution is given by adding a multiple of $(-q,p)$ to the coefficient vector (note that  $-q(p) + p(q) = 0$ is a primitive solution since $p$ and $q$ are relatively prime).  Thus, the solutions with the smallest absolute values of the $x$--coordinate to the unital equation are the one above and 
$$-\frac{q+1}{2} p +\frac{p+1}{2} q =1.$$  That is, the smallest possible value for $(x-x')$ is $x-x' = \frac{q-1}{2}$.  But, since $x$ and $x'$ both are nonnegative and less than $\frac{q-1}{2}$, this is impossible. 
As an example, if $p = 21$ and $q=23$, (so $g = 220$) we have the solutions $$11(21) - 10(23)=1$$ and $$-12(21) + 11(23) =1.$$  with $g = 220$.  We also have $x(21) + y(23) \le 220$ which imply that $x \le 220 /21$, so $0 \le x \le 10$.  Similarly for $x'$, so it is not possible for $|x - x'| = 11$.
Finally, we consider the case $i = j$.  Thus, our coefficients would satisfy $$(x-x')p + (y-y')q =0.$$ This implies that $x - x' $ is a multiple of $q$.  But this would imply that they are equal, since under our assumptions, both are nonnegative and also $xp \le pq - p -q +1 \le pq$, so $x < q$ and $x' <q$.
\vskip.1in
In summary, if we write  the Alexander polynomial of the $T_{p,q}$ torus knot, with $q-p=2$ as $\pm 1$ as $a_0 + \sum_{i=1}^g  a_i(t^i +t^{-i})$, then for $i \le \frac{p-1}{2}$, we have shown that  $a_i = (-1)^i$.
\end{proof}
\vskip.1in


\vfill \eject

\end{document}